\documentclass[12pt,final]{article}

\usepackage[utf8]{inputenx}
\usepackage[T1]{fontenc}

\usepackage{lmodern}           
\usepackage[scaled]{beramono}  
\usepackage[final]{microtype}  
\usepackage{siunitx}           
\usepackage{listings}          
\usepackage{url}               
\usepackage{amssymb}    
\usepackage{amsthm}     
\usepackage{thmtools}   
\usepackage{mathtools}  
\usepackage{mathrsfs}   
\usepackage{dsfont}     
\usepackage{cancel}     
\usepackage{stmaryrd}   

\usepackage{xspace}         
\usepackage[top=3.5cm,bottom=4cm]{geometry}

\usepackage[backend = biber, style = alphabetic]{biblatex}

\usepackage{varioref}
\usepackage{hyperref}
\usepackage[nameinlink, capitalize, noabbrev]{cleveref}

\usepackage{tikz-cd}
\usepackage{enumitem}
\usepackage{todonotes}

\declaretheorem[style = plain, numberwithin = section]{theorem}
\declaretheorem[style = plain,      sibling = theorem]{corollary}
\declaretheorem[style = plain,      sibling = theorem]{lemma}
\declaretheorem[style = plain,      sibling = theorem]{proposition}

\declaretheorem[style = definition, sibling = theorem]{definition}

\declaretheorem[style = remark,     sibling = theorem]{remark}
\declaretheorem[style = remark,     sibling = theorem]{claim}

\crefname{observation}{Observation}{Observations}
\Crefname{observation}{Observation}{Observations}
\crefname{conjecture}{Conjecture}{Conjectures}
\Crefname{conjecture}{Conjecture}{Conjectures}
\crefname{notation}{Notation}{Notations}
\Crefname{notation}{Notation}{Notations}
\Crefname{claim}{Claim}{Claims}
\Crefname{question}{Question}{Questions}

\newcommand{\CY}{Calabi--\kern-0.37exYau\xspace}

\newcommand{\ie}{\leavevmode\unskip, i.e.,\xspace}
\newcommand{\eg}{\leavevmode\unskip, e.g.,\xspace}

\newcommand{\Z}{\mathbb{Z}}    
\newcommand{\Q}{\mathbb{Q}}    
\newcommand{\R}{\mathbb{R}}    
\newcommand{\C}{\mathbb{C}}    
\renewcommand{\P}{\mathbb{P}}  

\DeclarePairedDelimiter{\set}{\lbrace}{\rbrace}        

\newcommand{\from}{\colon}

\newcommand{\sO}{\mathscr{O}}
\newcommand{\sE}{\mathscr{E}}

\newcommand{\sR}{\mathscr{R}}

\newcommand{\sL}{\mathscr{L}}

\newcommand{\sS}{\mathcal{S}}

\newcommand\restr[2]{{
  \left.\kern-\nulldelimiterspace 
  #1 
  \vphantom{\big|} 
  \right|_{#2} 
}}

\newcommand{\insec}{\cdot}
\newcommand{\NE}{\overline{\text{NE}}}

\newcommand{\product}{\Pi}

\DeclareMathOperator{\Pic}{Pic}

\DeclareMathOperator{\Sym}{Sym}
\DeclareMathOperator{\Gr}{Gr}

\DeclareMathOperator{\rk}{rk}

\usepackage{titlesec}
\titleformat*{\section}{\large\bfseries}
\titleformat*{\subsection}{\bfseries}
\titleformat*{\subsubsection}{\small\bfseries}

\title{\vspace{-2.4cm}\Large{\bfseries{Curve Classes on \CY Complete Intersections in Toric Varieties}}}
\author{Bjørn Skauli}
\date{}
\begin{document}
\maketitle

\begin{abstract}
  We prove the Integral Hodge Conjecture for curve classes on smooth varieties of dimension at least three constructed as a complete intersection of ample hypersurfaces in a smooth projective toric variety, such that the anticanonical divisor is the restriction of a nef divisor. In particular, this includes the case of smooth anticanonical hypersurfaces in toric Fano varieties. In fact, using results of Casagrande and the toric MMP, we prove that in each case, $H_2(X,\Z)$ is generated by classes of rational curves.
\end{abstract}

\section{Introduction}
On a smooth complex projective variety of dimension $n$, the vector space $H^k(X,\C)$ admits a Hodge decomposition into subspaces $H^{p,q}(X,\C)$, with $p+q = k$. The integral Hodge classes $H^{k,k}(X,\Z)$ are the classes in  $H^{2k}(X,\Z)$ which map to $H^{k,k}(X,\C)$ under the natural map
\[H^{2k}(X,\Z) \to H^{2k}(X,\C),\]
 and the class of any algebraic subvariety is an integral Hodge class. The Integral Hodge Conjecture asks whether the classes of algebraic subvarieties generate the integral Hodge Classes as a group.

A basic result in this direction is the Lefschetz (1,1)-theorem. This theorem states that the Integral Hodge Conjecture holds for codimension 1 classes. By the Hard Lefschetz theorem, this also implies that  Hodge Conjecture holds for degree $2n-2$ classes \ie classes of algebraic curves generate $H^{n-1,n-1}(X,\Q)$ as a vector space. 

However the Integral Hodge Conjecture might still fail for $H^{n-1,n-1}(X,\Z)$. Much work has been done exploring how this failure might occur, especially through constructing counterexamples to the Integral Hodge Conjecture for degree $2n-2$ classes. There are two ways in which the Integral Hodge Conjecture can fail and there are counterexamples illustrating both. The first way is through torsion classes in $H^{k,k}(X,\Z)$. Any torsion class is an integral Hodge class, and one can find counterexamples to the Integral Hodge Conjecture for curves by finding a torsion class in $H^{k,k}(X,\Z)$ which is not algebraic. In fact, the first counterexample to the Integral Hodge conjecture was of this form. In \cite{AH62}, Atiyah and Hirzebruch construct a projective variety with a degree 4 torsion class that is nonalgebraic.

The Integral Hodge Conjecture for curves can even fail modulo torsion. In \cite{Kol90}, Koll\'ar constructs counterexamples on projective hypersurfaces in $\P^4$ of high degree, on which there is a nontorsion, nonalgebraic class in $H^{n-1,n-1}(X,\Z)$.

On the other hand, by imposing restrictions on the geometry of the variety $X$, many positive results in the direction of the Integral Hodge Conjecture have also been found. In \cite{Voi06}, Voisin proves that for a complex projective threefold $X$ that is either uniruled or satisfies $K_X = \sO_X$ and $H^2(X,\sO_X) = 0$ the Integral Hodge Conjecture for curves holds. In \cite{Tot19}, Totaro shows more generally that it holds for all threefolds of Kodaira dimension 0 with $H^0(X,\sO(K_X)) \neq 0$. In \cite{BO20}, Benoist and Ottem construct a threefold $X$ such that $2K_X = 0$, and $X$ does not satisfy the Integral Hodge Conjecture, which shows that there is an important difference between assuming Kodaira dimension 0 and assuming that the canonical divisor is trivial. In \cite{Voi06}, Voisin also raises the question of whether the Integral Hodge Conjecture for curves holds for rationally connected varieties.

One reason for the interest in the Integral Hodge Conjecture for curves is to construct stable birational invariants of smooth projective varieties. Voisin introduced the group $Z^{2n-2} = H^{n-1,n-1}(X,\Z)/H^{n-1,n-1}(X,\Z)_{alg}$, (see \cite{SV05}, \cite{Voi06}, \cite{CV12} and \cite{Voi16}) which is a stable birational invariant, and is the trivial group for rational varieties. There are also other cases where the Integral Hodge Conjecture for varieties with trivial canonical divisor can give answers to other geometric questions. For instance in \cite{Voi17} Voisin relates the question of stable rationality of a cubic threefold to the question of whether a particular class in the intermediate Jacobian, an abelian variety of dimension 5, is algebraic.

In this paper, we will prove that the Integral Hodge Conjecture for curves holds on certain \CY varieties constructed as smooth complete intersections in smooth projective toric Fano varieties. The result will in fact hold more generally when the anticanonical divisor of the complete intersection $X$ is the restriction of a nef divisor on the ambient variety.
The only condition on the dimension of $X$ is that it must be at least 3. The main result is:
\begin{theorem}
	\label{thm:CompleteIntersectionIntroduction}
	Let $Y$ be a smooth projective toric variety and let $X \subset Y$ be a smooth complete intersection of ample hypersurfaces $H_1,\dots,H_k$, with $\dim X$ at least 3. Assume furthermore that $-K_Y - \sum_{i=1}^k H_i$ is nef on $Y$, so in particular $-K_X$ is nef. Then the Integral Hodge Conjecture for curves holds for $X$. More precisely, $H_2(X,\Z)$ is generated by classes of rational curves.
\end{theorem}
In the process of proving this theorem, we will also show:
\begin{proposition}
	\label{prop:SemigroupGenerated}
	Let $Y$ be a smooth projective toric variety and let $X \subset Y$ be a smooth complete intersection of ample hypersurfaces $H_1,\dots,H_k$, with $\dim X$ at least 3. Assume furthermore that $-K_Y - \sum_{i=1}^k H_i$ is nef on $Y$. Then the semigroup of effective curve classes on $Y$ is generated over $\Z$ by rational curves contained in the complete intersection $X$.
\end{proposition}

The main challenge in proving the Integral Hodge Conjecture for curves is finding algebraic representatives of generators of the group $H_2(X,\Z)$. In \cite{Cas03}, Casagrande proves that for the ambient toric variety $Y$, the group $H_2(Y,\Z)$ is generated by the classes of so-called contractible curves, which are algebraic. So we will prove that $H_2(X,\Z)$ also contains algebraic representatives of classes of contractible curves. The proof of \cref{thm:CompleteIntersectionIntroduction} is inspired by an argument given by Koll\'ar in an appendix to \cite{Bor91}, where he proves that for an anticanonical hypersurface $X$ in a Fano variety $Y$, the cones of effective curves $\NE(X)$ and $\NE(Y)$ coincide.

The structure of the paper is as follows: We first recall the main definitions and results used in this paper, in particular the results from \cite{Cas03} in \cref{sec:Preliminaries}. Then in \cref{sec:CompleteIntersection} we will give a proof of \cref{thm:CompleteIntersectionIntroduction}. 
The main result in \cite{Cas03} is that on smooth projective toric varieties, the semigroup of effective curve classes is generated over $\Z$ by contractible classes. To prove \cref{thm:CompleteIntersectionIntroduction} we will prove that, with assumptions as in the theorem, the complete intersection contains curves in each contractible class. To prove this, we will for a given contractible class construct a vector bundle such that the zero set of a section corresponds to curves of the given contractible class contained in $X$. We will then use ampleness of the hypersurfaces defining $X$ to check that the top Chern class of this bundle is nonzero.

\subsection*{Acknowledgements}
I would like to thank my advisor John Christian Ottem for suggesting this question, and for his guidance in writing this paper and patience throughout the writing process. I am also grateful to the anonymous referee for pointing out mistakes in an earlier version, and for comments greatly improving the exposition in the paper.

\section{Preliminaries}
\label{sec:Preliminaries}

Let $X$ be a smooth projective variety over $\C$ of dimension $n$. The \emph{integral Hodge classes}  $H^{k-1,k-1}(X,\Z)$ are the classes in $H^{2k}(X,\Z)$ that map to the subspace $H^{k,k}(X,\C)$ of the Hodge decomposition of $H^{2k}(X,\C)$ under the natural map 
\[H^{2k}(X,\Z) \to H^{2k}(X,\C).\]
We will write $H^{2k}(X,\Z)_{alg}$ for the subgroup of $H^{2k}(X,\Z)$ generated by classes of algebraic subvarieties of $X$. The Integral Hodge Conjecture asks if any integral Hodge class is a linear combination of classes of algebraic varieties, in other words if $H^{2k}(X,\Z)_{alg} = H^{k,k}(X,\Z)$.

We will focus on the Integral Hodge Conjecture for curves, which is the statement that $H^{n-1,n-1}(X,\Z)$ is generated by the classes of algebraic curves contained in $X$. Recall that on a smooth, projective toric, or more generally rational variety $Y$, $\dim H^i(Y,\sO_Y) = 0$ for $i > 0$, and $H_2(Y,\Z)$ is torsion free. As a consequence,
\[ H^{n-1,n-1}(Y,\Z) \simeq H_2(Y,\Z).\]
By the Lefschetz Hyperplane theorem, the same is true for an ample hypersurface $X \subset Y$ of dimension at least 3.

We also recall the definition of the Neron-Severi space $N_1(X)$, the space of 1-cycles modulo numerical equivalence:
\begin{definition}
Let $X$ be a smooth projective complete variety. We define the vector space
\[N_1(X) \coloneqq \set{\text{1-cycles in X}}/\equiv_{num} \otimes \R. \]
\end{definition}

The varieties we study in this paper are constructed starting from toric varieties. For a general introduction to toric varieties, one can see \eg \cite{CLS11}. We will recall some facts about toric varieties that we will use throughout.

A toric variety $Y$ corresponds to a fan $\Sigma$ in a real vector space $N_\R$, and many geometric properties of $Y$ are encoded by combinatorial properties of $\Sigma$. On a smooth, projective toric variety $Y$ defined by a fan $\Sigma$, the group generated by curve classes up to  numerical equivalence is isomorphic to the integral relations between primitive generators $x_i$ of the rays of $\Sigma$. The relations corresponding to torus invariant curves are called \emph{wall relations}. Furthermore, the intersection number $D_i \cdot C$ between a torus-invariant divisor $D_i$, corresponding to the ray spanned by $x_i$, and a curve $C$ corresponding to a wall relation $a_1x_1 + \cdots + a_{n+1}x_{n+1} = 0$ is the coefficient $a_i$ (thus 0 if the generator of the ray does not occur in the wall relation). Here $n$ is the dimension of $N_\R$, which also equals $\dim{Y}$. We will also need the observation that  since the anticanonical divisor $-K_Y$ is the sum of all the torus-invariant divisors, the intersection number $-K_Y \insec C$ equals the sum of all coefficients in the wall relation.

\begin{proposition}[{\cite[Proposition 6.4.1]{CLS11}}]
  Let $\Sigma$ be a simplicial fan in $N_\R$ with convex support of full dimension. Then there are dual exact sequences:
\[ 0 \to M_\R \xrightarrow{\alpha} \R^{\Sigma(1)} \xrightarrow{\beta} \Pic(Y_\Sigma) \otimes \R \to 0 \]
\[ 0 \to N_1(Y_\Sigma) \xrightarrow{\beta^*} \R^{\Sigma(1)} \xrightarrow{\alpha^*} N_\R \to 0 \]
where
\[ \alpha^*(e_\rho) = u_\rho \]
\[ \beta^*([C]) = (D_\rho \insec C)_{\rho \in \Sigma(1)}.\]
We write $\Sigma(1)$ for the rays of $\Sigma$, $M_\R$ is the dual space to $N_\R$, $e_\rho$ the standard basis vectors of $\R^{\Sigma(1)}$, $u_\rho$ the primitive generator of the ray $\rho \in \Sigma(1)$ and $C \subset Y_\Sigma$ a complete irreducible curve.
\end{proposition}

  In the case where $Y$ is a smooth projective toric variety, $N_1(Y)$ is isomorphic to $H_2(Y,\Z) \otimes \R$ and $H_2(Y,\Z)$ embeds into $N_1(Y)$. Furthermore, on a smooth projective toric variety, $H_2(Y,\Z)$ is generated by the classes of torus-invariant curves. This is a special consequence of a theorem of Jurkiewicz and Danilov \cite[Theorem 12.4.4]{CLS11}

\subsection{Contractible Classes of a Toric Variety}
Because $H_2(X,\Z)$ embeds into $N_1(X)$, we can use tools from the Minimal Model Program to study the question of the Integral Hodge Conjecture, in particular the results of Casagrande (see \cite{Cas03}).

\begin{definition}[{\cite[Definition 2.3]{Cas03}}]
  A primitive curve class $\gamma \in H_2(Y,\Z)$, where $Y$ is a complete, smooth toric variety, is called \emph{contractible} if there exists an equivariant toric morphism $\pi \from Y \to Z$ with connected fibers such that for every irreducible curve $C \subset Y$,
\[ \pi(C) = \set{pt} \iff [C] \in \Q_{\geq 0}\gamma . \]
\end{definition}
Recall that a class $\gamma \in H_2(Y,\Z)$ is \emph{primitive} if it is not a positive integer multiple of any other class. We will call a curve $C \subset Y$ \emph{contractible} if its class in  $H_2(Y,\Z)$ is a contractible class. In particular, a contractible curve will always have a class that is primitive in $H_2(Y,\Z)$.

The structure of a contraction of a  contractible class is described by the following result.

\begin{proposition}[{\cite[Corollary 2.4]{Cas03} }]
  \label{prop:ContractionStructure1}
  Let $Y$ be a smooth complete toric variety of dimension $n$,  $\gamma \in \NE(Y)$ a contractible class and $\pi \from Y \to Z$ the associated contraction.

Suppose first that $\gamma$ is nef, so that its wall relation is:
\[x_1 + \cdots + x_e = 0. \]
Then $Z$ is smooth of dimension $n-e+1$ and $\pi \from Y \to Z$ is a $\P^{e-1}$-bundle.

Suppose now that $\gamma$ is not nef, so that its wall relation is:
\[x_1 + \cdots + x_e - a_1 y_1 - \cdots - a_r y_r = 0 \, \, r > 0 .\]
Then $\pi$ is birational, with exceptional loci $E \subset Y$, $B \subset Z$, $\dim E = n-r$, $\dim B = n-e-r + 1$ and $\restr{\pi}{E} \from E \to B$ is a $\P^{e-1}$-bundle.
\end{proposition}
 By $\P^{e-1}$-bundle we mean a bundle that is locally trivial in the Zariski topology. In particular, there is a vector bundle $\sE$ on $B$ such that $E = \P(\sE)$.
\begin{remark}
  If $\gamma$ is nef \ie $\gamma \insec D \geq 0$ for all divisors $D$, then from how intersection numbers can be computed from the wall relation corresponding to $\gamma$, the wall relation can not have any negative coefficients. The positive coefficients in a wall relation corresponding to a  contractible curve are all equal to 1. It must therefore have the form described in the theorem. This happens precisely when curves of class $\gamma$ move to cover the entire toric variety.
\end{remark}
In contrast to contractions of extremal rays, if $\pi \from Y \to Z$ is a contraction of a contractible class, the target variety $Z$ is not necessarily projective even if $Y$ is projective. In fact, $Z$ is projective if and only if the contraction is a contraction of an extremal ray. However, the exceptional locus of the contraction of a contractible class has the structure of a projective bundle over a projective variety.


The reason we wish to consider contractible classes, as opposed to only extremal ones is the following result by Casagrande, which says that these rays generate the subgroup $H_2(Y,\Z)_{alg}$, the subgroup of $H_2(Y,\Z)$ generated by classes of algebraic curves.
\begin{theorem}[{\cite[Theorem 4.1]{Cas03}}]
  \label{thm:CasagrandeSpan}
  Let $Y$ be a smooth projective toric variety. Then for every $\eta \in H_2(Y,\Z)_{alg} \cap NE(Y)$ there is a decomposition:
\[ \eta = m_1 \gamma_1 + \cdots + m_r\gamma_r \]
with $\gamma_i$ contractible and $m_i \in \Z_{>0}$ for all $i=1,\dots,r$.
\end{theorem}
As an immediate consequence of this and the fact that $H_2(Y,\Z)_{alg} = H_2(Y,\Z)$, we see that $H_2(Y,\Z)$ is also generated by the classes of contractible curves.

\subsection{Ample Vector Bundles and Positivity of Chern classes}
We recall some basic definitions and central results on Chern classes of ample vector bundles, which will be useful later. Throughout the paper, we will use the convention that a projective bundle $\P(\sE)$ parametrizes one-dimensional quotients of $\sE$.


The central fact we will use is that nef (ample) vector bundles have effective (and nonzero) Chern classes.
\begin{theorem}[{\cite[Theorem 8.2.1]{Laz04II},\cite[Corollary 8.2.2]{Laz04II} }]
  	\label{thm:ChernPositive}
	Let $X$ be an irreducible projective variety or scheme of dimension $n$ and let $E$ be a nef vector bundle on $X$. Then
	\[ \int_X c_n(E) \geq  0\]
	The same statement holds of $E$ is replaced by a nef $\Q$-twisted bundle $E\!<\! \delta \!>$. If $E$ is ample the inequality is strict.
\end{theorem}
It is a straightforward consequence of this theorem that for a nef vector bundle and $j \leq n$, we have $c_j(E) \geq 0$, with strict inequality if the bundle is ample.



\section{Complete intersections}
\label{sec:CompleteIntersection}
Let $Y$ be a smooth projective toric variety and $H_1, \dots, H_k$ ample hypersurfaces, such that 
\[X \coloneqq H_1 \cap \cdots \cap H_k\]
is a smooth complete intersection. Furthermore, we will assume that $X$ has dimension at least 3, hence $Y$ must have dimension at least $k+3$. Under these assumptions, a generalization of the Lefschetz Hyperplane theorem (See \cite[Remark 3.1.32]{Laz04II}) shows that 
\begin{equation}
  \label{eq:LefschetzHyperplane}
  H_2(X,\Z) \simeq H_2(Y,\Z)
\end{equation}

Together with \cref{thm:CasagrandeSpan}, this suggests a strategy for proving the Integral Hodge Conjecture on complete intersections of ample hypersurfaces in a smooth projective toric variety. If we can prove that $X$ contains a representative of each contractible class, these curve classes generate $H_2(Y,\Z)$ by \cref{thm:CasagrandeSpan}. Hence they generate $H_2(X,\Z)$ by \eqref{eq:LefschetzHyperplane}, so the Integral Hodge Conjecture holds for $X$.

When $X$ is a hypersurface and the contraction of the contractible class is a $\P^1$-bundle, the following result by Koll\'ar from the appendix to \cite{Bor91} is an example of this strategy. 
\begin{lemma}
	\label{lem:KollarBundle}
	Let $B$ be a normal projective variety.
	\begin{enumerate}[label = \emph{(\roman*)}]
		\item Let $g \from E \to B$ be a $\P^1$-bundle. Let $X \subset E$ be a subvariety such that $\restr{g}{X} \from X \to B$ is finite of degree 1. If $X$ is ample, then $\dim B \leq 1$
		
		\item Let $g \from E \to B$ be a conic bundle. Let $X$ be a subvariety such that $\restr{g}{X} \from X \to B$ is finite of degree $\leq 2$. If $X$ is ample, then $\dim B \leq 2$.
\end{enumerate}
\end{lemma}

Clearly, if the restriction $\restr{g}{X}$ is not finite, then $X$ contains a fiber of $g$, which is a contractible curve. Koll\'ar proves \cref{lem:KollarBundle} by constructing a vector bundle $\sE$ such that if $\restr{g}{X}$ is finite then the top Chern class of $\sE$ is zero. Ampleness, together with \cref{thm:ChernPositive}, then gives a contradiction if $\dim B \geq 1$. We will use a similar idea to prove \cref{thm:CompleteIntersectionIntroduction}. First we relate contractible curves on $X$ to Chern classes of a vector bundle, and then we use \cref{thm:ChernPositive} to prove that the Chern classes are nonzero. However, straightforwardly applying \cref{thm:ChernPositive} will not be sufficient, since the relevant vector bundles are  nef, but not necessarily ample.

\subsection{Setup}
Fix a contractible class $[C]$ in the toric ambient variety $Y$. Let $E = \P(\sE) \to B$ be the exceptional locus of the contraction of $[C]$ on $Y$. The goal is to prove that $X$ contains a curve $C$ of this class, which is the class of a line in a fiber of $\pi \from E \to B$. To do this, we construct and study a vector bundle on the relative Grassmannian $\Gr(2,\sE)$, which defines the relative Fano scheme of lines of the complete intersection $X$.  

Let the complete intersection $X$ be defined by the intersection of the ample hypersurfaces $H_1, \dots, H_k$. For $i= 1,\dots,k$, $d_i \geq 1$ are integers, and $\sL_i$ are line bundles on $B$, such that the line bundle $\sO_{\P(\sE)}(H_i)$ is isomorphic to $\sO_{\P(\sE)}(d_i) \otimes \pi^*\sL_i$. The lines in the fibers of $\P(\sE) \to B$ are parametrized by the relative Grassmannian $\Gr(2,\sE)$, with projection map $p \from \Gr(2,\sE) \to B$. We will denote the tautological rank 2 subbundle on $\Gr(2,\sE)$ by $\sS \subset p^*\sE$.

\begin{definition}
\label{def:MDefinition}
With notation as above, for a complete intersection $X$ and contraction of a contractible class $[C]$, we define the following vector bundle on $\Gr(2,\sE)$.
\begin{equation}
  \label{eq:MDefinition}
  \mathscr{M}_{X,C} \coloneqq \bigoplus_{i=1}^k \Sym^{d_i} \sS^* \otimes p^*\sL_i.
\end{equation}
We use the subscripts when we wish to indicate the dependence on the contraction of $[C]$, and the line bundles $\sO(H_i)$, where $H_i$ are the hypersurfaces defining $X$.
\end{definition}

This is a bundle on $\Gr(2,\sE)$ of rank $r = \sum_{i=1}^k (d_i+1)$, and is a quotient of
\[\bigoplus_{i=1}^k \Sym^{d_i}\left(p^*\sE \right) \otimes p^*\sL_i = p^* \left(\bigoplus_{i=1}^k \Sym^{d_i}\sE \otimes \sL_i \right). \]
Furthermore, the complete intersection $X = H_1 \cap \cdots \cap H_k$ induces a section of $\mathscr{M}_{X,C}$. The section of $\mathscr{M}_{X,C}$ induced by $X$ vanishes precisely at the lines in fibers of $\pi \from E \to B$ contained in $X$. Therefore, if the top Chern class $c_r(\mathscr{M}_{X,C})$ is nonzero, then any section, and in particular the section induced by $X$, must vanish at some point of $\Gr(2,n+1)$. So our goal in this section is to prove that $c_r(\mathscr{M}_{X,C})$ is nonzero.

It will be useful that this vector bundle has the following positivity property.
\begin{lemma}
  \label{lem:NefPerturbation}
  Assume $\sO(d_i) \otimes \pi^*\sL_i$, $i=1,\dots,k$ are ample line bundles on $\pi \from \P(\sE) \to B$. Let $\mathscr{M}$ be the vector bundle $\bigoplus_{i=1}^k \Sym^{d_i} \sS^* \otimes p^*\sL_i$ on the relative Grassmannian $p \from \Gr(2,\sE) \to B$. Then for any line bundle $\mathscr{A}$ on $B$,  $\mathscr{M} \otimes \epsilon p^* \mathscr{A}$ is nef for all sufficiently small positive $\epsilon$.
\end{lemma}
\begin{proof}
  We first consider the case $k=1$. After a suitable $\Q$-twist of $\sE$, we may assume that $\mathscr{L}_1 = \sO_B$, and since $\sO_{\P(\sE)}(d_1) \otimes \pi^*\mathscr{L}_1 = \sO_{\P(\sE)}(d_1)$ is ample, the vector bundle $\sE$ will be ample as well. For this $\sE$, the bundle $\mathscr{M}$ is equal to $\Sym^{d_1}(\sS^*)$, so $\mathscr{M} \otimes \epsilon p^*\mathscr{A} = \Sym^{d_1}(\sS^*) \otimes \epsilon p^*\mathscr{A}$, which is a quotient of $p^*\left(\Sym^{d_1}\mathscr{E}\right) \otimes \epsilon p^*\mathscr{A}$. Since $\sE$ is an ample vector bundle on $B$, so is $\Sym^{d_1}(\mathscr{E})$. Hence, for any line bundle $\mathscr{A}$ on $B$ and all sufficiently small positive $\epsilon$, $\Sym^{d_1}\mathscr{E} \otimes \epsilon \mathscr{A}$ is an ample $\Q$-vector bundle on $B$. The vector bundle 
\[p^* \left( \Sym^{d_1}\mathscr{E} \otimes \epsilon \mathscr{A} \right)\]
is a pullback of an ample vector bundle, hence nef. Since $\mathscr{M} \otimes \epsilon p^*\mathscr{A}$ is a quotient of this bundle, it is also nef.

For $k > 1$, the argument for $k=1$ shows that for any line bundle $\mathscr{A}$ on $B$ and all $i = 1,\dots,k$, the vector bundle $\Sym^{d_i}(\sS^*) \otimes p^*\mathscr{L}_i \otimes \epsilon_i p^* \mathscr{A}$ is nef for all sufficiently small $\epsilon_i$. So for $\epsilon$ sufficiently small, $\Sym^{d_i}(\sS^*) \otimes p^*\mathscr{L}_i \otimes \epsilon p^* \mathscr{A}$ is nef for all $i$. Hence $\mathscr{M}$ is a direct sum of nef vector bundles and therefore nef.
\end{proof}
The special case $\mathscr{A} = \sO_B$ shows that in particular $\mathscr{M}_{X,C}$ is nef.

\subsection{Rank and Dimension}
The first thing to check is that the rank of $\mathscr{M}_{X,C}$ is less than or equal to the dimension of $\Gr(2,\sE)$ so it is possible for the top Chern class of $\mathscr{M}_{X,C}$ to be nonzero. We will see that when the divisor $-K_Y - \sum_{i=1}^k H_i$ intersects the contractible curve nonnegatively, it imposes bounds on the relevant dimensions and degrees. These bounds ensure that the rank of $\mathscr{M}_{X,C}$ is at most $\dim \Gr(2,\sE)$.

\begin{proposition}
	\label{prop:SumDegreesBound}
	Let $H_1, \dots, H_k$ be ample divisors in a smooth projective toric variety $Y$, and let $\pi \from E = \P(\sE) \to B$ be the exceptional locus of the contraction of a contractible class $[C]$. Let $d_i$ be integers, and $\mathscr{L}_i$ be line bundles on $B$, chosen such that $\sO_{\P(\sE)}(H_i) \simeq \sO(d_i) \otimes \pi^*\mathscr{L}_i$. Assume furthermore that $(-K_Y - \sum_{i=1}^k H_i) \insec C \geq 0$. Then we have the following inequality: 
	\begin{equation}
		\label{eq:CIDegreeBound}
		\dim F + \dim E \geq \dim Y + \sum_{i=1}^k d_i -1
	\end{equation}
	
\end{proposition}
\begin{proof}
  	 Let
  	\[x_1 + \cdots + x_e - a_1 y_1 - \cdots - a_r y_r =0 \]
  	be the wall relation corresponding $C$,
  	where the $x_i$ and $y_j$ are the generators of rays in the fan of $Y$. Using \cref{prop:ContractionStructure1} we find that \eqref{eq:CIDegreeBound} is equivalent to the inequality:
	\[e-1 + n - r  \geq n + \sum_{i=1}^r d_i - 1. \]
	It therefore suffices to prove the inequality $e-r \geq \sum_{i=1}^rd_i$. Note that
	\[ e - \sum_{i=1}^r a_i = -K_Y \insec C = \sum_{j=1}^k d_j + (-K_Y - \sum_{i=1}^k H_i) \insec C. \]
	By assumption, $(-K_Y - \sum_{i=1}^k H_i) \insec C \geq 0$, hence $e-\sum_i a_i \geq \sum_j d_j$. Since the $a_i$ are positive integers, we get $e-r \geq e-\sum_i a_i \geq \sum_j d_j$ as desired.
\end{proof}

\begin{remark}
	In \cite{Wis91}, Wi\'{s}niewski proves a similar inequality for extremal contractions of smooth, not necessarily toric, varieties.
\end{remark}

We can now find conditions such that the rank of the bundle $\mathscr{M}_{X,C}$ in \eqref{eq:MDefinition} does not exceed the dimension of $\Gr(2,\sE)$.
\begin{corollary}
	\label{cor:NonNegExpDim}
	Let $X = H_1 \cap \cdots \cap H_k$, be a complete intersection of ample divisors in a smooth projective toric variety $Y$ of dimension $n$, with $n \geq k+3$. Let $\pi \from  E = \P(\sE) \to B$ be the exceptional locus of the contraction of a contractible class $[C]$, and assume that $(-K_Y - \sum_{i=1}^kH_i) \insec C \geq 0$.Then the bundle $\mathscr{M}_{X,C}$ from \cref{def:MDefinition} has rank at most $\dim \Gr(2,\sE)$.
\end{corollary}
\begin{proof}
	The rank of $\mathscr{M}_{X,C}$ is $\sum_{i=1}^k(d_i+1)$, so we must prove the inequality:
	\begin{equation}
		\label{eq:CIDegreeBound2}
		\sum_{i=1}^k (d_i + 1) \leq \dim \Gr(2,\sE) = 2(\rk \sE - 2) + \dim B
	\end{equation}
        We wish to apply \cref{prop:SumDegreesBound}. Using \cref{prop:ContractionStructure1} we have, in the notation from \eqref{eq:CIDegreeBound} \[\dim \Gr(2,\sE) = \dim F + \dim E - 2.\]
         This gives the chain of inequalities
         \begin{align*}
         	\dim \Gr(2,\sE) &= \dim F + \dim E - 2 \geq \dim Y + \sum_{i=1}^k d_i - 3 \\
         	&\geq 3+k+\sum_{i=1}^k d_i - 3 = \sum_{i=1}^k(d_i+1)
         \end{align*}
        where the first inequality follows from \eqref{eq:CIDegreeBound} and the second uses that by assumption $\dim X \geq 3$, hence $\dim Y \geq 3+k$.
\end{proof}
 In other words, when $X = H_1 \cap \cdots \cap H_k$ is a complete intersection of ample hypersurfaces of dimension at least 3, and $\P(\sE) \to B$ a contraction of a contractible class $[C]$ such that $(-K_Y - \sum_{i=1}^kH_i) \insec C \geq 0$, a dimension estimate leads us to expect that $X$ contains curves of class $[C]$. In fact, for general choices of the $H_i$, the Fano scheme parametrizing these curves will have expected dimension. We will prove this here under the assumption that the Fano scheme is nonempty. This assumption holds by \cref{prop:PermutationPositive} in the next section.

\begin{proposition}
	\label{prop:ExpectedDimCI}
	Assume that $H_1,H_2,\dots,H_k$ are ample and general in their respective linear systems. Assume further that the relative Fano scheme of the intersection $X \coloneqq H_1 \cap \cdots \cap H_k$, with respect to a given contraction with exceptional locus $E = \P(\sE) \to B$, is nonempty. Then the relative Fano scheme of $X$ has the expected dimension.
\end{proposition}
\begin{proof}
	Let $V \coloneqq H^0(\sO_{\P(\sE)}(H_1)) \times \cdots \times H^0(\sO_{\P(\sE)}(H_k))$, and let $I$ be the incidence correspondence $I \coloneqq \set{ (l,f_1,\dots,f_k) \vert l \subset (f_1 = \cdots = f_k = 0)} \subset \Gr(2,\sE) \times V$. We first check that $I$ has the expected codimension, $\sum_{i=1}^k(d_i+1)$. We can check this by considering the fiber of $I \to \Gr(2,\sE)$, which has codimension $\sum_{i=1}^k(d_i+1)$ in $V$. Since by assumption the projection $I \to V$ is dominant, the general fiber of this projection must have codimension $\sum_{i=1}^k(d_i+1)$ in $\Gr(2,\sE)$.
\end{proof}

\subsection{Positivity}

Our goal is to prove the following proposition, showing that the top Chern class of $\mathscr{M}_{X,C}$ from \cref{def:MDefinition} is not merely effective, but in fact nonzero.
\begin{proposition}
	\label{prop:PermutationPositive}
	Let $E = \P(\sE) \to B$ be the exceptional locus of a contraction of a contractible class $[C]$ on a smooth projective toric variety $Y$ of dimension $n \geq 3+k$. Let $X = H_1 \cap \cdots \cap H_k$ be a complete intersection of ample divisors on $Y$. Let $\mathscr{M}_{X,C}$ be as in \cref{def:MDefinition}, and assume that $(-K_Y - \sum_{i=1}^k H_i) \insec C \geq 0$. Then the top Chern class $c_r(\mathscr{M}_{X,C})$ is nonzero and effective, where $r=\sum_{i=1}^k(d_i+1)$ is the rank of $\mathscr{M}_{X,C}$.
\end{proposition}
Before we prove this result, we need some preliminary results on Chern classes of symmetric powers. In particular, we will to prove that the Chern classes of the $d$-th symmetric power of the rank 2 tautological subbundle on $\Gr(2,n+1)$ are effective and nonzero.
\begin{lemma}
	\label{prop:EffectiveSymProdChern}
	Let $\sR$ be a rank 2 bundle. Then for any symmetric power $\Sym^d \sR$, with $d \geq 2$, for $j \leq d$ the $j-$th Chern class of $\Sym^d \sR$ is of the form
	\[c_j(\Sym^d \sR) = a_j c_1(\sR)^j + P_j(c_1(\sR),c_2(\sR)) \]
	where $P_j$ is a polynomial with nonnegative integral coefficients and $a_j > 0$ for $j\leq d$. The top Chern class $c_{d+1}(\Sym^d \sR)$ is of the form:
	\[c_{d+1}(\Sym^d \sR) = a_{d+1} c_1(\sR)^{d-1}c_2(\sR) + P_{d+1}(c_1(\sR),c_2(\sR))\]
	where $P_{d+1}$ is a polynomial with nonnegative integral coefficients and $a_{d+1} > 0$.
\end{lemma}
\begin{proof}
	Let $\alpha,\beta$ be the Chern roots of $\sR$. If $d$ is odd, the Chern polynomial $c(\Sym^d \sR)$ is given by:
	\begin{align*}
		&\product_{i=0}^d (1+(d-i)\alpha + i\beta) \\
		&= \product_{i<\frac{d}{2}} (1+(d-i)\alpha+i\beta)(1+i\alpha + (d-i)\beta) \\
		&= \product_{i<\frac{d}{2}} (1 + d(\alpha + \beta) + i(d-i)(\alpha + \beta)^2 + (d-2i)^2 \alpha \beta)\\
		&= \product_{i<\frac{d}{2}} (1 + dc_1(\sR) + i(d-i) c_1(\sR)^2 + (d-2i)^2 c_2(\sR))
	\end{align*}
	If $d$ is even, the Chern polynomial  $c(\Sym^d \sR)$ is given by:
	\begin{align*}
		&\product_{i=0}^d (1+(d-i)\alpha + i\beta)\\
		&= (1+\frac{d}{2} \alpha + \frac{d}{2}\beta )\product_{i<\frac{d}{2}} (1+(d-i)\alpha+i\beta)(1+i\alpha + (d-i)\beta)\\
		&= (1+\frac{d}{2} \alpha + \frac{d}{2}\beta ) \product_{i<\frac{d}{2}} \left( 1 + d(\alpha + \beta) + i(d-i)(\alpha + \beta)^2 + (d-2i)^2 \alpha \beta \right) \\
		&= (1+\frac{d}{2} c_1(\sR) )\product_{i<\frac{d}{2}} \left( 1 + dc_1(\sR) + i(d-i) c_1(\sR)^2 + (d-2i)^2 c_2(\sR) \right)
	\end{align*}
	From this description we see that as long as $j < d+1$, $c_j(\Sym^d\sR)$ will have the form:
	\[c_j(\Sym^d \sR) = a_j c_1(\sR)^j + P_j(c_1(\sR),c_2(\sR)) \]
	with $a_j > 0$ for $j < d+1$, and all coefficients of $P_j$ nonnegative integers. The top Chern class will be of the form:
	\[c_{d+1}(\Sym^d \sR) = a_{d+1} c_1(\sR)^{d-1}c_2(\sR) + P_{d+1}(c_1(\sR),c_2(\sR))\]
	with $a_{d+1} > 0$ and all coefficients of $P_{d+1}$ nonnegative integers.
\end{proof}

Recall that if $X$ is an $N$-dimensional variety, we call a class $\alpha \in H^{2k}(X,\Z)$ is called nef if it has nonnegative intersection with the class of every $(N-k)$-dimensional subvariety of $X$. Using \cref{prop:EffectiveSymProdChern}, we can give the following description of the Chern classes of $\Sym^d \sS^*$.
\begin{lemma}
	\label{prop:SubbundleChernClasses}
	Let $Gr(2,n+1)$ be the Grassmannian of lines in $\P^n$, let $\sS$ be the rank 2 tautological subbundle. Then for $1 \leq j < d+1$, the Chern class $c_j(\Sym^d \sS^*)$ can be written as the sum of two terms
\[c_j(\Sym^d \sS^*) = a_{j}(c_1(\sS^*))^j + \alpha_j, \]
and $c_1(\sS^*)$ is the class of an ample divisor, and $\alpha_j$ is a nef and effective class. The top Chern class $c_{d+1}(\Sym^d \sS^*)$ can be written as
\[c_{d+1}(\Sym^d \sS^*) = a_{d+1}(c_1(\sS^*))^{d-1}c_2(\sS^*) + \alpha_{d+1}, \]
with $a_{d+1} > 0$, where $\alpha_{d+1}$ is a nef and effective class. 
\end{lemma}

\begin{proof}
	The Chern class of $\sS^*$ is the sum of Schubert cycles: $c(\sS^*) = 1 + \sigma_1 + \sigma_{11}$. The class $\sigma_1$ is an ample divisor class on $\Gr(2,n+1)$ since it is the pullback of a hyperplane via the Plücker embedding. Furthermore, $\sigma_{11}$ is the Schubert cycle of lines contained in a hyperplane $H \subset \P^n$. It follows that on the Grassmannian of lines in $\P^n$, any monomial in $\sigma_1$ and $\sigma_{11}$ is a nef and effective cycle.

 We then apply \cref{prop:EffectiveSymProdChern} to see that for $j \leq d$, 
	\[c_j(\Sym^d \sS^*) = a_j c_1(\sS^*)^j + \text{effective and nef cycles},\] with $a_j >0$, and the top Chern class 
	\[c_{d+1}(\Sym^d \sS^*) = a_{d+1}c_1(\sS^*)^{d-1}c_2(\sS^*) + \text{effective and nef cycles}.\]
\end{proof}

\begin{corollary}
	\label{cor:SumSubbundleChern}
	All Chern classes of the bundle $\bigoplus_{i=1}^k \Sym^{d_i}(\sS^*)$ are effective on $\Gr(2,n+1)$ the Grassmannian of lines in projective space. Furthermore, if $0 \leq j \leq \min(\sum_{i=1}^k (d_i +1), \dim \Gr(2,n+1))$, the Chern class $c_j(\bigoplus_{i=1}^k \Sym^{d_i}(\sS^*))$ is nonzero
\end{corollary}
\begin{proof}
	The Chern polynomial of a direct sum is the product of the Chern polynomial of the summands. \cref{prop:SubbundleChernClasses} describes the form of these Chern polynomials. In particular, we see that the $j$-th Chern classes contain a term of the form $bc_1(\sS^*)^\alpha c_2(\sS^*)^\beta$, with $\alpha + 2\beta = j$, and the coefficient $b$ is strictly greater than $0$. Furthermore, if $\alpha + 2\beta = j \leq \dim \Gr(2,n+1)$, we see that $c_1(\sS^*)^\alpha c_2(\sS^*)^\beta > 0$ by computing with the relevant Schubert cycles. 
\end{proof}

It follows from these results that when restricted to a fiber of $p \from \Gr(2,\sE) \to B$, the Chern classes of $\mathscr{M}_{X,C}$ are strictly positive, unless they vanish for dimensional reasons.

With this we are ready to prove \cref{prop:PermutationPositive} using a perturbation argument, based on \cref{lem:NefPerturbation}

\begin{proof}[Proof of \cref{prop:PermutationPositive}]
Set $b \coloneqq \dim B$, and recall that $r = \sum_{i=1}^k(d_i+1)$ is the rank of the bundle $\mathscr{M}_{X,C}$. Let $D' \subset B$ be a smooth ample divisor, and set $D \coloneqq p^*D'$. By \cref{lem:NefPerturbation}, $\mathscr{M}_{X,C} \otimes -\epsilon D$ remains a nef vector bundle for sufficiently small $\epsilon$. So $\mathscr{M}_{X,C} \otimes -\epsilon D$ has effective Chern classes. The top Chern class of $\mathscr{M}_{X,C} \otimes -\epsilon D$ can be expressed as
\[c_r(\mathscr{M}_{X,C} \otimes - \epsilon D) = c_r(\mathscr{M}_{X,C}) - \epsilon D \insec c_{r-1}(\mathscr{M}_{X,C}) + \cdots + (-1)^b\epsilon^bD^b \insec c_{r-b}(\mathscr{M}_{X,C}).\]
Assume for contradiction that $c_r(\mathscr{M}_{X,C}) = 0$. Then since $D \insec c_{r-1}(\mathscr{M}_{X,C})$ is effective, we must have $D \insec c_{r-1}(\mathscr{M}_{X,C}) = 0$. Otherwise, $c_r(\mathscr{M}_{X,C} \otimes -\epsilon D)$ would not be effective for some small $\epsilon$, contradicting nefness of $\mathscr{M}_{X,C} \otimes -\epsilon D$. Let $B_1$ be the hypersurface $D'$. Then we must have $c_{r-1}(\restr{\mathscr{M}_{X,C}}{p^{-1}(B_1)}) = 0$. If $\dim B_1 = 0$, then $p^{-1}(B_1)$ is a union of fibers $F_1,\dots,F_N$ of $p$, and by \cref{cor:NonNegExpDim}, $r-1 \leq \dim(F_i)$. So on each fiber $F_i$, $\restr{M}{F_i}$ has strictly positive Chern classes by \cref{cor:SumSubbundleChern}. Hence $c_{r-1}(\restr{\mathscr{M}_{X,C}}{p^{-1}(B_1)})$ must also be strictly positive. This gives our contradiction.

If $\dim B_1 \geq 1$, we repeat the argument. We find that
\begin{align*}
   c_{r-1}(\restr{\mathscr{M}_{X,C}}{B_1} \otimes - \epsilon D) = c_{r-1}(\restr{\mathscr{M}_{X,C}}{B_1}) & - \epsilon D \insec c_{r-2}(\restr{\mathscr{M}_{X,C}}{B_1}) + \cdots \\ &+ (-1)^{b-1}\epsilon^{b-1}D^{b-1} \insec c_{r-b}(\restr{\mathscr{M}_{X,C}}{B_1}) \geq 0 
\end{align*}
for all sufficiently small $\epsilon$. In particular, we must have $D \insec c_{r-2}(\restr{\mathscr{M}_{X,C}}{B_1}) = 0$. So $c_{r-2}(\restr{\mathscr{M}_{X,C}}{p^{-1}(B_2)})$ must be 0, where $B_2 \subset B_1$ is a smooth subvariety representing the divisor $\restr{D'}{B_1}$. Repeating this construction if necessary, eventually we reach either $c_{r-b}((\restr{\mathscr{M}_{X,C}}{p^{-1}(B_b)}))$ if $r > b$, where $B_b$ is nonempty and has dimension 0, or $c_0((\restr{\mathscr{M}_{X,C}}{p^{-1}(B_r)})$ if $r \leq b$. So if $c_r(\mathscr{M}_{X,C})=0$, we must also have $c_{r-b}((\restr{\mathscr{M}_{X,C}}{p^{-1}(B_b)})) = 0$ or $c_0((\restr{\mathscr{M}_{X,C}}{p^{-1}(B_r)}) = 0$. The latter is impossible. In the former case, we conclude from \cref{cor:NonNegExpDim} that since $r \leq \dim(\Gr(2,\sE))$, also $r-b \leq \dim(p^{-1}(B_b))$. We can therefore apply \cref{cor:SumSubbundleChern} and find that $c_{r-b}((\restr{\mathscr{M}_{X,C}}{B_b}))$ is strictly positive. Since this is a contradiction, we conclude that $c_r(\mathscr{M}_{X,C})$ must be effective and nonzero.

\end{proof}

Using \cref{prop:PermutationPositive}, we get a condition for when a complete intersection of ample hypersurfaces $H_i$ of dimension at least 3 in a smooth projective toric variety contains a curve of a given contractible class.
\begin{corollary}
	\label{cor:NegativeCurveContained}
	Let $X = H_1 \cap \cdots \cap H_k$ be a smooth complete intersection of ample hypersurfaces in a smooth projective toric variety $Y$, and assume $\dim X \geq 3$.
Let $Y \to Z$ be the contraction of a contractible class $[C]$, with exceptional locus $E = \P(\sE) \to B$. If $(-K_Y - \sum_{i=1}^k H_i) \insec C \geq 0$, then $X$ contains a curve with class $[C]$ in $N_1(X)$. 
\end{corollary}
\begin{proof}
The complete intersection $X$ induces a section of the bundle 
	\[\mathscr{M}_{X,C} = \bigoplus_{i=1}^k \Sym^{d_i}(\sS^*) \otimes p^*\mathscr{L}_i\]
	 on $\Gr(2,\sE) \to B$. From \cref{prop:PermutationPositive} we see that the top Chern class is effective and nonzero. So any section of $\mathscr{M}_{X,C}$ must vanish at some point. Since the section of $\mathscr{M}_{X,C}$ induced by $X$ vanishes precisely at curves of class $[C]$ contained in $X$, we may conclude. 
\end{proof}
If the divisors $H_i$ defining $X$ are general in their respective linear systems, then by \cref{prop:ExpectedDimCI}, the contractible curves contained in $X$ are parametrized by a space of expected dimension. If the $H_i$ are not general, the space of contractible curves contained in $X$ will have at least expected dimension, but is potentially larger.
\begin{remark}
	Compare this result to \cite[Theorem 4.3]{HLW02}, which applies to the similar setting where $X \subset Y$ is a smooth ample divisor in a smooth variety of dimension at least 4, and $C$ is an extremal curve class with $-K_X \insec C \geq 0$. Then \cite[Theorem 4.3]{HLW02} implies that $X$ contains a curve whose class in $N_1(X)$ is some multiple of $C$.
\end{remark}

\cref{thm:CompleteIntersectionIntroduction} follows easily from  \cref{cor:NegativeCurveContained}.
\begin{theorem}[{= \cref{thm:CompleteIntersectionIntroduction}}]
	\label{thm:CompleteIntersection}
	Let $Y$ be a smooth projective toric variety, an let $X \subset Y$ be a smooth complete intersection of ample hypersurfaces $H_1,\dots,H_k$, with $\dim X$ at least 3. Assume furthermore that $-K_Y - \sum_{i=1}^kH_i$ is nef. Then the Integral Hodge Conjecture for curves holds for $X$, and in fact $H_2(X,\Z)$ is generated by classes of rational curves.
\end{theorem}

\begin{proof}
	By assumption $-K_Y - \sum_{i=1}^kH_i$ is nef, so for any contractible class $[C]$ in $H_2(Y,\Z)$, the hypotheses of \cref{cor:NegativeCurveContained} is satisfied. Thus, $X$ contains representatives of all contractible classes. Since the classes of contractible curves span $H_2(X,\Z)$ by \cref{thm:CasagrandeSpan}, we can conclude that the Integral Hodge Conjecture holds for $X$. Since all contractible curves are rational, the second statement also follows.
\end{proof}
More precisely, this proves \cref{prop:SemigroupGenerated} from the Introduction.
\subsection{Example}
\label{sec:ExampleLineBlowup}
 We end with an example, which illustrates \cref{thm:CompleteIntersection} and its proof in a simple case. 
	Let $l \subset \P^4$ be a torus-invariant line, and let $Y$ be the blow-up of $\P^4$ along $l$. The Picard group of $Y$ is generated by $H$, the pullback of the hyperplane class in $\P^4$ and the exceptional divisor $E$. The homology group $H_2(Y,\Z)$ is generated by $h$, the pullback of a general line in $\P^4$ and $e$, the class of a line in a positive-dimensional fiber of the blow-up map.
	The cone of curves on $Y$ is generated by the extremal classes $e$ and $h-e$ , with contractions given by the blow-up map $b \from Y \to \P^4$ and the resolution of the projection of $\P^4$ from $l$, $p \from Y \to \P^2$ respectively.
	
	The toric variety $Y$ is smooth, projective and Fano, with anticanonical divisor $5H - 2E$. Let $X$ be a general smooth anticanonical hypersurface. Then $X$ is the strict transform of a quintic hypersurface containing the line $l$ with multiplicity two. By the Lefschetz Hyperplane Theorem, $H_2(X,\Z) \simeq H_2(Y,\Z)$, which is generated by the extremal curve classes on $Y$. To check the Integral Hodge Conjecture for Curves on $X$ we should therefore check that $X$ contains representatives of each of these curve classes. We will look at each of these in turn.
	
	\begin{claim}
		The anticanonical hypersurface $X$ contains a curve with class $e$.
	\end{claim}
	\begin{proof}
		Consider the exceptional locus $E$ of the blow-up map $b \from Y \to \P^4$. Then $E$ is isomorphic to $\P^2 \times \P^1$, with the blow-up map as the second projection. As a divisor, $X$ restricts to a divisor of class $\sO(2,3)$ on $E$. On each fiber of this bundle, $X$ restricts to a plane conic, and some of these plane conics will be reducible. In fact, one can check that there will be 9 reducible fibers when $X$ is general. A line in any of these reducible fibers is a curve in $X$ of class $e$.
	\end{proof}
	
	\begin{claim}
		The anticanonical hypersurface $X$ contains a curve with class $h-e$.
	\end{claim}
	\begin{proof}
		Consider the extremal contraction $p \from Y \to \P^2$. This gives $Y$ the structure of the projective bundle $p \from \P(\sO_{\P^2}^{\oplus 2} \oplus \sO_{\P^2}(1)) \to \P^2$, with fibers isomorphic to $\P^2$. If we denote by $\zeta$ the divisor corresponding to the line bundle $\sO_{\P(\sO_{\P^2}^{\oplus 2} \oplus \sO_{\P^2}(1))}(1)$,  $X$ is linearly equivalent to the divisor $3\zeta + 2p^*H_{\P^2}$, and restricts to cubic curves on the fibers of $p$. For a plane cubic curve to contain a line is a codimension 2 condition, so we would expect $X$ to contain lines in the fibers of $p$. In fact, one can check that a general anticanoncial hypersurface $X$ will contain 234 lines in fibers of $p$. This line is a curve of class $h-e$ contained in $X$.
	\end{proof}
\printbibliography

\typeout{get arXiv to do 4 passes: Label(s) may have changed. Rerun}
\end{document}